\documentclass[a4paper,10pt]{amsart}

\usepackage[utf8]{inputenc} 
\usepackage{amsmath}
\usepackage{amsbsy}
\usepackage{amssymb}
\usepackage{amscd,amsthm}
\usepackage{verbatim}
\usepackage[english]{babel}
\usepackage{dsfont}
\usepackage{anysize}
\usepackage{hyperref}

\newcommand{\e}{\varepsilon}

\newcommand{\N}{\mathds{N}}
\newcommand{\R}{\mathds{R}}

\newcommand{\p}{\phantom}
\newcommand{\q}{\quad}

\newtheorem{thm}{Theorem}[section]
\newtheorem{lem}[thm]{Lemma}
\newtheorem{kor}[thm]{Corollary}
\newtheorem{conj}[thm]{Conjecture}
\newtheorem{prop}[thm]{Proposition}
\newtheorem*{bp}{Bertrand's Postulate}

\theoremstyle{definition}

\theoremstyle{remark}
\newtheorem*{rema}{Remark}

\title{On the number of primes up to the $n$th Ramanujan prime}
\author{Christian Axler}
\address{Institute of Mathematics\\ Heinrich-Heine University Düsseldorf\\ 40225 Düsseldorf, Germany}
\email{christian.axler@hhu.de}
\date{\today}

\begin{document}

\begin{abstract}
The $n$th Ramanujan prime is the smallest positive integer $R_n$ such that for all $x \geq R_n$ the interval $(x/2, x]$ contains at least $n$ primes. In this 
paper we undertake a study of the sequence $(\pi(R_n))_{n \in \mathds{N}}$, which tells us where the $n$th Ramanujan prime appears in the sequence of all 
primes. In the 
first part we establish new explicit upper and lower bounds for the number of primes up to the $n$th Ramanujan prime, which imply an asymptotic formula for 
$\pi(R_n)$ conjectured by Yang and Togb\'e. In the second part of this paper, we use these explicit estimates to derive a result concerning an inequality 
involving $\pi(R_n)$ conjectured by of Sondow, Nicholson and Noe.
\end{abstract}

\maketitle

\section{Introduction}

Let $\pi(x)$ denotes the number of primes not exceeding $x$.
In 1896, Hadamard \cite{hadamard1896} and de la Vall\'{e}e-Poussin \cite{vallee1896} proved, independently, the asymptotic formula $\pi(x) \sim x/\log x$ as $x 
\to \infty$, which is known as the \textit{Prime Number Theorem}. Here, $\log x$ is the natural logarithm of $x$. In his later paper \cite{vallee1899}, where he 
proved the existence of a zero-free region for the Riemann zeta-function $\zeta(s)$ to the left of the line $\text{Re}(s) = 1$, de la Vall\'{e}e-Poussin also 
estimated the error term in the Prime Number Theorem by showing
\begin{equation}
\pi(x) = \frac{x}{\log x} + O \left( \frac{x}{\log^2x} \right). \tag{1.1} \label{1.1}
\end{equation} 
The prime counting function and the asymptotic formula \eqref{1.1} play an important role in the definition of Ramanujan primes, which have their origin in 
Bertrand's postulate.

\begin{bp}
For each $n \in \N$ there is a prime number $p$ with $n < p \leq 2n$.
\end{bp}

In terms of the prime counting function, Bertrand’s postulate states that $\pi(2n) - \pi(n) \geq 1$ for every $n \in \N$. Bertrand’s postulate was first proved 
by Chebyshev \cite{chebyshev} in 1850. In 1919, Ramanujan \cite{ramanujan} proved an extension of Bertrand’s postulate by investigating inequalities of the form 
$\pi(x) - \pi(x/2) \geq n$ for $n \in \N$. In particular, he found that
\begin{displaymath}
\pi(x) - \pi \left( \frac{x}{2} \right) \geq 1 \q (\text{respectively 2, 3, 4, 5, \ldots})
\end{displaymath}
for every
\begin{equation}
x \geq 2 \q (\text{respectively 11, 17, 29, 41, \ldots}). \tag{1.2} \label{1.2}
\end{equation}
Using the fact that $\pi(x) - \pi(x/2) \to \infty$ as $x \to \infty$, which follows from \eqref{1.1}, Sondow \cite{sondow} introduced the notation $R_n$ to 
represent the smallest positive integer for which the inequality $\pi(x) - \pi(x/2) \geq n$ holds for every $x \geq R_n$ . In \eqref{1.2}, Ramanujan calculated 
the numbers $R_1 = 2$, $R_2 = 11$, $R_3 = 17$, $R_4 = 29$, and $R_5 = 41$. All these numbers are prime, and it can easily be shown that $R_n$ is actually prime 
for every $n \in \N$. In honor of Ramanujan’s proof, Sondow \cite{sondow} called the number $R_n$ the $n$th Ramanujan prime. A legitimate question is, where 
the $n$th Ramanujan prime appears in the sequence of all primes. Letting $p_k$ denotes the $k$th prime number, we have $R_n = p_ {\pi(R_n)}$, and it seems 
natural to study the sequence $(\pi(R_n))_{n \in \N}$. The first few 
values of $\pi(R_n)$ for $n=1,2,3, \ldots$ are
\begin{displaymath}
\pi(R_n) = 1,5,7,10,13,15,17,19,20,25,26,28,31,35,36,39,41,42,49,50,51,52,53, \ldots.
\end{displaymath}
For further values of $\pi(R_n)$, see \cite{sloane}. Since both $R_n$ for large $n$ and $\pi(x)$ for large $x$ are hard to compute, we are interested in 
explicit upper and lower bounds for $\pi(R_n)$.
Sondow \cite[Theorem 2]{sondow} found a first lower bound for $\pi(R_n)$ by showing that the inequality
\begin{equation}
\pi(R_n) > 2n \tag{1.3} \label{1.3}
\end{equation}
holds for every positive integer $n \geq 2$. Combined with \cite[Theorem 3]{sondow} and the Prime Number Theorem, we get the asymptotic relation
\begin{equation}
\pi(R_n) \sim 2n \q\q (n \to \infty). \tag{1.4} \label{1.4}
\end{equation}
This, together with \eqref{1.3}, means, roughly speaking, that the probability of a randomly chosen prime being a Ramanujan prime is slight less than 
$1/2$. The first upper bound for $\pi(R_n)$ is also due to Sondow \cite[Theorem 2]{sondow}. He found that the upper bound $\pi(R_n) < 4n$ holds for every 
positive integer $n$, and conjectured \cite[Conjecture 1]{sondow} that the inequality $\pi(R_n) < 3n$ holds for every positive integer $n$. This conjecture was 
proved by Laishram \cite[Theorem 2]{laishram} in 2010. Applying Theorem 4 from the paper of Sondow, Nicholson and Noe \cite{sondownicholsonnoe}, we get a 
refined upper bound for the number of primes less or equal to $\pi(R_n)$, namely that the inequality $\pi(R_n) \leq \pi(41p_{3n}/47)$ holds for every positive 
integer $n$ with equality at $n = 5$. Srinivasan \cite[Theorem 1.1]{srinivasan} proved that for every $\e > 0$ there exists a positive integer $N = N(\e)$ such 
that
\begin{equation}
\pi(R_n) < \lfloor 2n(1+\e) \rfloor \tag{1.5} \label{1.5}
\end{equation}
for every positive integer $n \geq N$ and conclude \cite[Corollary 2.1]{srinivasan} that $\pi(R_n) \leq 2.6n$ for every positive integer $n$. The present 
author \cite[Theorem 3.22]{axler2016} showed independently that for each $\e  > 0$ there is a computable positive integer $N = N(\e)$ so that $\pi(R_n) \leq 
\lceil 2n(1+\e) \rceil$ for every positive integer $n \geq N$ and conclude that
\begin{equation}
\pi(R_n) \leq \lceil tn \rceil \tag{1.6} \label{1.6}
\end{equation}
for every positive integer $n$, where $t$ is a arbitrary real number satisfying $t > 48/19$. The inequality \eqref{1.5} was improved by Srinivasan and 
Nicholson \cite[Theorem 1]{srinivasannicholson}. They proved that
\begin{displaymath}
\pi(R_n) \leq 2n \left( 1 + \frac{3}{\log n + \log \log n - 4} \right)
\end{displaymath}
for every positive integer $n \geq 242$. Later, Srinivasan and Ar\'es \cite[Theorem 1.1]{srinivasanares} found a more precise result by showing that for every 
$\e > 0$ there exists a positive integer $N = N(\e)$ such that
\begin{equation}
\pi(R_n) < 2n \left( 1 + \frac{\log 2 + \e}{\log n + j(n)} \right) \tag{1.7} \label{1.7}
\end{equation}
for every positive integer $n \geq N$, where $j$ is any positive function satisfying $j(n) \to \infty$ and $nj'(n) \to 0$ as $n \to \infty$. Setting $\e = 
0.5$ and $j(n) = \log \log n - \log 2 - 0.5$, they found \cite[Corollary]{srinivasanares} that the inequality \eqref{1.7} holds for every positive integer $n 
\geq 44$. In 2016, Yang and Togb\'e \cite[Theorem 1.2]{yangtogbe} established the following current best upper and lower bound for $\pi(R_n)$ when $n$ satisfies 
$n > 10^{300}$.

\begin{prop}[Yang, Togb\'e] \label{prop101}
Let $n$ be a positive integer with $n > 10^{300}$. Then
\begin{displaymath}
\beta < \pi(R_n) < \alpha,
\end{displaymath}
where
\begin{align*}
\alpha & = 2n \left( 1 + \frac{\log 2}{\log n} - \frac{\log 2 \log \log n - \log^22 - \log 2 - 0.13}{\log^2n}\right), \\
\beta & = 2n \left( 1 + \frac{\log 2}{\log n} - \frac{\log 2 \log \log n - \log^22 - \log 2 + 0.11}{\log^2n}\right).
\end{align*}
\end{prop}

The proof of Proposition \ref{prop101} is based on explicit estimates for the $k$th prime number $p_k$ obtained by Dusart \cite[Proposition 6.6 and Proposition 
6.7]{dusart2010} and on Srinivasan's lemma \cite[Lemma 2.1]{srinivasan} concerning Ramanujan primes. Instead of using Dusart's estimates, we use the estimates 
obtained in \cite[Corollary 1.2 and Corollary 1.4]{axler2017} to get the following improved upper bound for $\pi(R_n)$.

\begin{thm} \label{thm102}
Let $n$ be a positive integer satisfying $n \geq 5\,225$ and let
\begin{equation}
U(x) = \frac{\log 2 \log x(\log \log x)^2 - c_1 \log x \log \log x + c_2 \log x - \log^22 \log \log x + \log^32 + \log^22}{\log^4x + 
\log^3x \log \log x - \log^3x \log 2 - \log^2x \log 2}, \tag{1.8} \label{1.8}
\end{equation}
where $c_1 = 2 \log^22 + \log 2$ and $c_2 = \log^32 + 2\log^22 + 0.565$. Then
\begin{displaymath}
\pi(R_n) < 2n \left( 1 + \frac{\log 2}{\log n} - \frac{\log 2 \log \log n - \log^22 - \log 2}{\log^2n} + U(n) \right).
\end{displaymath}
\end{thm}

With the same method, we used for the proof of Theorem \ref{thm102}, we get the following more precised lower bound for the number of primes not exceeding the 
$n$th Ramanujan prime.

\begin{thm} \label{thm103}
Let $n$ be a positive integer satisfying $n \geq 1\,245$ and let
\begin{equation}
L(x) = \frac{\log 2 \log x(\log \log x)^2 - d_1 \log x \log \log x + d_2 \log x - \log^22 \log \log x + \log^32 + \log^22}{\log^4x + 
\log^3x \log \log x - \log^3x \log 2 - \log^2x \log 2}, \tag{1.9} \label{1.9}
\end{equation}
where $d_1 = 2 \log^22 + \log 2 + 1.472$ and $d_2 = \log^32 + 2\log^22 - 2.51$. Then
\begin{displaymath}
\pi(R_n) > 2n \left( 1 + \frac{\log 2}{\log n} - \frac{\log 2 \log \log n - \log^22 - \log 2}{\log^2n} + L(n)\right).
\end{displaymath}
\end{thm}

A direct consequence of Theorem \ref{thm102} and Theorem \ref{thm103} is the following result, which implies the correctness of a conjecture stated by Yang and 
Togb\'e \cite[Conjecture 5.1]{yangtogbe} in 2015.

\begin{kor} \label{kor104}
Let $n \geq 2$ be a positive integer. Then
\begin{displaymath}
\pi(R_n) = 2n \left( 1 + \frac{\log 2}{\log n} - \frac{\log 2 \log \log n - \log^22 - \log 2}{\log^2n} + \frac{\log 2(\log \log n)^2}{\log^3 n} + O\left( 
\frac{\log \log n}{\log^3n} \right) \right).
\end{displaymath}
\end{kor}

The initial motivation for writing this paper, was the following conjecture stated by Sondow, Nicholson and Noe \cite[Conjecture 1]{sondownicholsonnoe} 
involving $\pi(R_n)$.

\begin{conj}[Sondow, Nicholson, Noe] \label{conj105}
For $m = 1,2,3, \ldots$, let $N(m)$ be given by the following table:
\begin{center}
\begin{tabular}{|l||c|c|c|c|c|c|c|c|}
\hline
$m$\rule{0mm}{4mm}    & $1$ & $   2$ & $  3$ & $  4$ & $ 5$ & $ 6$ & $7,8, \ldots, 19$ & $20,21,\ldots$ \\ \hline
$N(m)$\rule{0mm}{4mm} & $1$ & $1245$ & $189$ & $189$ & $85$ & $85$ & $             10$ & $           2$ \\ \hline
\end{tabular} .
\end{center}
Then we have
\begin{equation}
\pi(R_{mn}) \leq m \pi(R_n) \q\q \forall \, n \geq N(m). \tag{1.10} \label{1.10}
\end{equation}
\end{conj}

Note that the inequality \eqref{1.10} clearly holds for $m = 1$ and every positive integer $n$. In the cases $m = 2, 3, \ldots, 20$, the inequality 
\eqref{1.10} has been verified for every positive integer $n$ with $R_{mn} < 10^9$. For any fixed positive integer $m$, we have, by \eqref{1.4}, $\pi(R_{mn}) 
\sim 2mn \sim m\pi(R_n)$ as $n \to \infty$. A first result in the direction of Conjecture \ref{conj105} is due to Yang and Togb\'e \cite[Theorem 
1.3]{yangtogbe}. They used Proposition \ref{prop101} to find the following result, which proves Conjecture \ref{conj105} when $n$ satisfies $n > 10^{300}$.

\begin{prop}[Yang, Togb\'e] \label{prop106}
For $m = 1,2,3, \ldots$, and $n > 10^{300}$, we have
\begin{displaymath}
\pi(R_{mn}) \leq m\pi(R_n).
\end{displaymath}
\end{prop}

Using the same method, we apply Theorem \ref{thm102} and Theorem \ref{thm103} to get the following result.

\begin{thm} \label{thm107}
The Conjecture \ref{conj105} of Sondow, Nicholson and Noe holds except for $(m,n) = (38,9)$.
\end{thm}

\section{Preliminaries}

Let $n$ be a positive integer. For the proof of Theorem \ref{thm102} and Theorem \ref{thm103}, we need sharp estimates for the $n$th prime number. The current 
best upper and lower bound for the $n$th prime number were obtained in \cite[Corollary 1.2 and Corollary 1.4]{axler2017} and are given as follows.

\begin{lem} \label{lem201}
For every positive integers $n \geq 46\,254\,381$, we have
\begin{displaymath}
p_n < n \left( \log n + \log \log n - 1 + \frac{\log \log n - 2}{\log n} - \frac{(\log \log n)^2-6\log \log n + 10.667}{2 \log^2 n} \right).
\end{displaymath}
\end{lem}

\begin{lem} \label{lem202}
For every positive integer $n \geq 2$, we have
\begin{displaymath}
p_n > n \left( \log n + \log \log n - 1 + \frac{\log \log n - 2}{\log n} - \frac{(\log\log n)^2-6\log \log n + 11.508}{2\log^2 n} \right).
\end{displaymath}
\end{lem}

\section{Proof of Theorem \ref{thm102}}

To prove Theorem \ref{thm102}, we use the method investigated by Yang and Togb\'e \cite{yangtogbe} for the proof of the upper bound for $\pi(R_n)$ given in 
Proposition \ref{prop101}. First, we note following result, which was obtained by Srinivasan \cite[Lemma 2.1]{srinivasan}.
Although it is a direct consequence of the definition of a Ramanujan prime, it plays an important role in the proof of the upper bound for $\pi(R_n)$ in 
Proposition 
\ref{prop101}.

\begin{lem}[Srinivasan] \label{lem301}
Let $R_n = p_s$ be the $n$th Ramanujan prime. Then we have $2p_{s-n} < p_s$ for every positive integer $n \geq 2$.
\end{lem}

Now, let $n$ be a positive integer. We define for each real $x$ with $2n < x < 2.6n$ the functions
\begin{equation}
G(x) = x \left( \log x + \log \log x - 1 + \frac{\log\log x - 2}{\log x} - \frac{(\log \log x)^2 - 6 \log \log x + 10.667}{2 \log ^2 x}\right) \tag{3.1} 
\label{3.1}
\end{equation}
and
\begin{equation}
H(x) = x \left( \log x + \log \log x - 1 + \frac{\log\log x - 2}{\log x} - \frac{(\log \log x)^2 - 6 \log \log x + 11.508}{2 \log ^2 x}\right), \tag{3.2} 
\label{3.2}
\end{equation}
and consider the function $F_1 : (2n, 2.6n) \to \R$ defined by
\begin{equation}
F_1(x) = G(x) - 2H(x-n). \tag{3.3} \label{3.3}
\end{equation}
In the following proposition, we note a first property of the function $F_1(x)$ concerning its derivative.

\begin{prop} \label{prop302}
Let $n$ be a positive integer with $n \geq 16$. Then $F_1(x)$ is a strictly decreasing function on the interval $(2n, 2.6n)$.
\end{prop}

\begin{proof}
Setting
\begin{displaymath}
q_1(x) = \frac{\log \log x-2}{\log x} - \frac{(\log \log x)^2 - 4 \log \log x + 4.667}{2\log^2x} + \frac{(\log \log x)^2 - 7 \log \log x + 13.667}{\log^3x}
\end{displaymath}
and
\begin{align*}
r_1(x) & = - \frac{2(\log \log(x-n)-1)}{\log(x-n)} + \frac{(\log \log(x-n))^2 - 4 \log \log(x-n) + 5.508}{2\log^2(x-n)} \\
& \p{\q\q} - \frac{2(\log \log(x-n))^2 - 14 \log \log(x-n) + 29.016}{\log^3(x-n)},
\end{align*}
a straightforward calculation shows that the derivative of $F_1(x)$ is given by
\begin{displaymath}
F_1'(x) = \log x - 2 \log(x-n) + \log \log x - 2 \log \log(x-n) + \frac{1}{\log x}
+ q_1(x) + r_1(x).
\end{displaymath}
Note that $\log \log(x-n) \geq 1$, $t^2 - 4t + 4.667 > 0$ and $2t^2 - 14t + 29.016 > 0$ for every $t \in \R$. Hence
\begin{align*}
F_1'(x) & < \log x - 2 \log(x-n) + \log \log x - 2 \log \log(x-n) + \frac{1}{\log x} + \frac{\log \log x-2}{\log x} \\
& \p{\q\q} + \frac{(\log \log x)^2 - 7 \log \log x + 13.667}{\log^3x} + \frac{(\log \log(x-n))^2 - 4 \log \log(x-n) + 5.508}{\log^2(x-n)}.
\end{align*}
The function $t \mapsto (\log \log t - 2)/\log t$ has a global maximum at $t = \exp(\exp(3))$. Together with $32 \leq 2n < x < 2.6n$, and the fact that the 
functions $t \mapsto ((\log \log t)^2 - 7 \log \log t + 13.667)/\log^3t$ and $t \mapsto ((\log \log t)^2 - 4 \log \log t + 5.508)/\log^3t$ are monotonic 
decreasing for every $t > 1$, we obtain that
\begin{displaymath}
F_1'(x) < 1.772 - \log n + \log \log(2.6n) - \log(\log^2 n).
\end{displaymath}
Finally, we use the fact that $t\log^2t > e^{1.772}\log(2.6t)$ for every $t \geq 6$ to get $F_1'(x) < 0$ for every $x \in (2n, 2.6n)$, which means that $F_1(x)$ 
is a strictly decreasing function on the interval $(2n, 2.6n)$.
\end{proof}

Next, we define the function $\gamma : \R_{\geq 4} \to \R$ by
\begin{equation}
\gamma(x) = \frac{\log 2 + \log 2/\log x + 0.565/\log^2x}{\log x + \log \log x - \log 2 - \log 2/\log x}. \tag{3.4} \label{3.4}
\end{equation}
A simple calculation shows that
\begin{equation}
\gamma(x) = \frac{\log 2}{\log x} - \frac{\log 2 \log \log x - \log^2 2 - \log 2}{\log^2 x} + U(x), \tag{3.5} \label{3.5}
\end{equation}
where $U(x)$ is defined as in \eqref{1.8}. In the following lemma, we note some useful properties of $\gamma(x)$.

\begin{lem} \label{lem303}
Let $\gamma(x)$ be defined as in \eqref{3.4}. Then the following hold:
\begin{enumerate}
 \item[(a)] $\gamma(x) > 0$ for every $x \geq 8$,
 \item[(b)] $\gamma(x) < \log 2/ \log x$ for every $x \geq 10\,734$,
 \item[(c)] $\gamma(x) < 1/4$ for every $x \geq 10\,734$. 
\end{enumerate}
\end{lem}

\begin{proof}
The statement in (a) is clear. To prove (b), we first note that $U(x) < \log 2 (\log \log x)^2/ \log^3x$ for every $x \geq 230 \geq \exp(\exp(1 + \log 2))$. Now 
we use \eqref{3.5} and the fact that $(\log \log x - \log 2 - 1) \log x \geq (\log \log x)^2$ for every $x \geq 10\,734$, to conclude (b). Finally, (c) is a 
direct consequence of (b).
\end{proof}

Now, we give a proof of Theorem \ref{thm102}.

\begin{proof}[Proof of Theorem \ref{thm102}]
First, we consider the case where $n$ is a positive integer with $n \geq 528\,491\,312 \geq \exp(\exp(3))$. By \eqref{1.3} and \eqref{1.6}, we have $2n < 
\pi(R_n) < 2.6n$. Hence $\pi(R_n) \geq 2n \geq 1\,056\,982\,624$ and $\pi(R_n)-n \geq 528\,491\,312$. Now we apply Lemma \ref{lem201} and Lemma \ref{lem202}
to get that $F_1(\pi(R_n)) > p_{\pi(R_n)} - 2p_{\pi(R_n)-n}$, where $F_1$ is defined as in \eqref{3.3}. Since $R_n = p_{\pi(R_n)}$, Srinivasan's Lemma 
\ref{lem301} yields
\begin{equation}
F_1(\pi(R_n)) > 0. \tag{3.6} \label{3.6}
\end{equation}
For convenience, we write in the following $\gamma = \gamma(n)$ and $\alpha = 2n(1+ \gamma)$. Now, by \eqref{3.5}, we 
need to show that $\pi(R_n) < \alpha$. For this, we first show that $F_1(\alpha) < 0$. By Lemma \ref{lem303}, we have $2n < \alpha < 2.6n$. Further,
\begin{equation}
\frac{F_1(\alpha)}{2n} = (1+ \gamma)\log 2 - \gamma\log n + \gamma + A_1 + B_1 + C_1 + D_1 + (1+2\gamma)\frac{0.841}{2\log^2(n+2n\gamma)}, \tag{3.7} \label{3.7}
\end{equation}
where
\begin{align*}
A_1 & = (1+ \gamma)\log(1+ \gamma) - (1+2\gamma)\log(1+2\gamma), \\
B_1 & = (1+ \gamma)\log \log(2n+ 2n\gamma) - (1+2\gamma)\log \log(n+2n\gamma), \\
C_1 & = (1+ \gamma) \frac{\log\log(2n+2n\gamma) - 2}{\log(2n+2n\gamma)} - (1+2\gamma)\frac{\log\log(n+2n\gamma) - 2}{\log(n+2n\gamma)}, \\
D_1 & = - (1+ \gamma) \frac{(\log \log(2n+2n\gamma))^2 - 6 \log \log(2n+2n\gamma) + 10.667}{2 \log^2(2n+2n\gamma)} \\ 
& \p{\q\q} + (1+2\gamma) \frac{(\log \log(n+2n\gamma))^2 - 6 \log \log(n+2n\gamma) + 10.667}{2 \log^2(n+2n\gamma)}.
\end{align*}
In the following, we give upper bounds for the quantities $A_1$, $B_1$, $C_1$ and $D_1$. We start with $A_1$. We use the inequalities
\begin{equation}
t - \frac{t^2}{2} < \log(1+t) < t, \tag{3.8} \label{3.8}
\end{equation}
which hold for every real $t > 0$, and Lemma \ref{lem303}(c) to get
\begin{equation}
A_1 < (1+\gamma) \gamma - (1+2\gamma)( 2\gamma - 2\gamma^2) = -\gamma - \gamma^2 + 4\gamma^3 < - \gamma. \tag{3.9} \label{3.9}
\end{equation}
Next, we estimate $B_1$. Using the right-hand side inequality of \eqref{3.8}, we easily get
\begin{equation}
B_1 < \frac{(1+\gamma) \log 2}{\log n} - \gamma \log \log n. \tag{3.10} \label{3.10}
\end{equation}
To find an upper bound for $C_1$, we note that $t \mapsto (\log \log t - 2)/\log t$ is a decreasing function on the interval $(\exp(\exp(3)), \infty)$. 
Together with Lemma \ref{lem303}(a), we obtain that the inequality
\begin{equation}
C_1 < 0 \tag{3.11} \label{3.11}
\end{equation}
holds. Finally, we estimate $D_1$. For this purpose, we consider the function $f : (1, \infty) \to \R$ defined by
\begin{displaymath}
f(x) = \frac{(\log \log x)^2 - 6 \log \log x + 10.667}{2 \log^2 x}.
\end{displaymath}
By the mean value theorem, there exists a real number $\xi \in (n + 2n\gamma, 2n+2n\gamma)$ such that $f(2n+2n\gamma) - f(n+2n\gamma) = nf'(\xi)$. Since 
$f''(x) \geq 0$ for every $x > 1$, we get $f'(\xi) \geq f'(n+2n\gamma) \geq f'(n)$. Hence we get
\begin{displaymath}
f(n+2n\gamma) - f(2n+2n\gamma) = - nf'(\xi) \leq -nf'(n) = \frac{(\log \log n)^2 - 7 \log \log n + 13.667}{\log^3 n}.
\end{displaymath}
Therefore
\begin{displaymath}
D_1 < (1+\gamma) \frac{(\log \log n)^2 - 7 \log \log n + 13.667}{\log^3 n} + \gamma f(n+2n\gamma).
\end{displaymath}
Since $f(x)$ is a strictly decreasing function on the interval $(1,\infty)$, it follows that the inequality
\begin{equation}
D_1 < (1+\gamma) \frac{(\log \log n)^2 - 7 \log \log n + 13.667}{\log^3 n} + \gamma \frac{(\log \log n)^2 - 6 \log \log n + 10.667}{2 \log^2 n} \tag{3.12} 
\label{3.12}
\end{equation}
holds. Combining \eqref{3.7} with \eqref{3.9}-\eqref{3.12}, we get
\begin{align*}
\frac{F_1(\alpha)}{2n} & < (1+ \gamma)\log 2 - \gamma\log n + \frac{(1+\gamma) \log 2}{\log n} - \gamma \log \log n + (1+\gamma) \frac{r_1(\log \log n)}{\log^3 
n} \\ 
& \p{\q\q} + \gamma \frac{r_2(\log \log n)}{2 \log^2 n} + (1+2\gamma)\frac{0.841}{2\log^2 n},
\end{align*}
where $r_1(t) = t^2 - 7t + 13.667$ and $r_2(t) = t^2 - 6t + 10.667$. The functions $t \mapsto r_1(\log \log t)/\log t$, $t \mapsto r_1(\log \log t)/\log^2 t$ 
and $t \mapsto r_2(\log \log t)/\log t$ are decreasing on the interval $(1,\infty)$. Hence $r_1(\log \log n) \leq r_1(3)$ and $r_2(\log \log n) \leq r_2(3)$. 
Together with Lemma \ref{lem303}(a), Lemma \ref{lem303}(b) and $n \geq \exp(\exp(3))$, we obtain that
\begin{displaymath}
\frac{F_1(\alpha)}{2n} < (1+ \gamma)\log 2 - \gamma\log n + \frac{(1+\gamma) \log 2}{\log n} - \gamma \log \log n + \frac{0.565}{\log^2n}.
\end{displaymath}
Now we use \eqref{3.4} to get that the right-hand side of the last inequality is equal to $0$. Hence $F_1(\alpha) < 0$. Together with $2n < \pi(R_n), \alpha < 
2.6n$, the inequality \eqref{3.6} and Proposition \ref{prop302}, we get $\pi(R_n) < \alpha$. We conclude by direct computation.
\end{proof}

We get the following weaker but more compact upper bounds for the parameter $s$.

\begin{kor} \label{kor304}
For every positive integer $n \geq 2$, we have
\begin{displaymath}
\pi(R_n) < 2n \left( 1 + \frac{\log 2}{\log n} - \frac{\log 2 \log \log n - \log^22 - \log 2}{\log^2n} + \frac{\log 2(\log \log n)^2}{\log^3n} \right).
\end{displaymath}
\end{kor}

\begin{proof}
If $n \geq 5\,225$, the corollary follows directly from Theorem \ref{thm102}, since $U(x) \leq \log 2 (\log \log x)^2/\log^3x$ for every $x \geq 230$. For the 
remaining cases of $n$, we use a computer.
\end{proof}

In the next corollary, we reduce the number $10^{300}$ in Proposition \ref{prop101} as follows.

\begin{kor} \label{kor305}
For every positive integer $n$ satisfying $n \geq 4\,842\,763\,560\,306$, we have
\begin{displaymath}
\pi(R_n) < 2n \left( 1 + \frac{\log 2}{\log n} - \frac{\log 2 \log \log n - \log^22 - \log 2 - 0.13}{\log^2n} \right).
\end{displaymath}
\end{kor}

\begin{proof}
Note that $U(x) \leq 0.13/\log^2x$ for every $x \geq 4\,842\,763\,560\,306$. Now we can use Theorem \ref{thm102}.
\end{proof}

\begin{kor} \label{kor306}
Let $n$ be a positive integer satisfying $n \geq 640$. Then
\begin{displaymath}
\pi(R_n) < 2n \left( 1 + \frac{\log 2}{\log n} \right).
\end{displaymath}
\end{kor}

\begin{proof}
For every positive integer $n \geq 10\,734$, we have
\begin{displaymath}
\frac{\log 2 \log \log n - \log^22 - \log 2}{\log^2n} - \frac{\log 2(\log \log n)^2}{\log^3n} > 0
\end{displaymath}
an it suffices to apply Corollary \ref{kor304}. We conclude by direct computation.
\end{proof}

\section{Proof of Theorem \ref{thm103}}

Using a simalar argument as in the proof of Lemma \ref{lem301}, Yang and Togb\'e \cite[p. 248]{yangtogbe} derived the following result.

\begin{lem}[Yang, Togb\'e] \label{lem401}
Let $R_n = p_s$ be the $n$th Ramanujan prime. Then we have $p_s < 2p_{s-n+1}$ for every positive integer $n$.
\end{lem}

Next, we define for each positive integer $n$ the function $F_2 : (2n, 2.6n) \to \R$ by
\begin{equation}
F_2(x) = H(x) - 2G(x-n+1), \tag{4.1} \label{4.1}
\end{equation}
where the functions $G(x)$ and $H(x)$ are given by \eqref{3.1} and \eqref{3.2}, respectively. In Proposition \ref{prop302}, we showed that for every positive 
integer $n \geq 16$, the function $F_1(x)$ is decreasing on the interval $(2n, 2.6n)$. In the following proposition, we get a similar result for the function 
$F_2(x)$.

\begin{prop} \label{prop402}
Let $n$ be a positive integer with $n \geq 15$. Then $F_2(x)$ is a strictly decreasing function on the interval $(2n, 2.6n)$.
\end{prop}

\begin{proof}
A straightforward calculation shows that the derivative of $F_2(x)$ is given by
\begin{align*}
F_2'(x) & = \log x - 2 \log(x-n+1) + \log \log x - 2 \log \log(x-n+1) + \frac{1}{\log x} + \frac{\log \log x-2}{\log x} \\
& \p{\q\q} - \frac{(\log \log x)^2 - 4 \log \log x + 5.508}{2\log^2x} + \frac{(\log \log x)^2 - 7 \log \log x + 14.508}{\log^3x} \\
& \p{\q\q} - \frac{2(\log \log(x-n+1)-1)}{\log(x-n+1)} + \frac{(\log \log(x-n+1))^2 - 4 \log \log(x-n+1) + 4.667}{2\log^2(x-n+1)} \\
& \p{\q\q} - \frac{2(\log \log(x-n+1))^2 - 14 \log \log(x-n+1) + 27.334}{\log^3(x-n+1)}.
\end{align*}
Now we argue as in the proof of Proposition \ref{prop302} to obtain that the inequality
\begin{displaymath}
F_2'(x) < 1.717 - \log n + \log \log(2.6n) - \log(\log^2 n)
\end{displaymath}
holds for every real $x$ such that $2n < x < 2.6n$. Since $t\log^2t > e^{1.717}\log(2.6t)$ for every $t \geq 6$, we get that $F_2(x)$ is 
a strictly decreasing function on the interval $(2n, 2.6n)$.
\end{proof}

Now, we define the function $\delta : \R_{\geq 4} \to \R$ by
\begin{equation}
\delta(x) = \frac{\log 2 + \log 2/\log x - (1.472\log \log x + 2.51)/\log^2x}{\log x + \log \log x - \log 2 - \log 2/\log x}. \tag{4.2} \label{4.2}
\end{equation}
A simple calculation shows that
\begin{equation}
\delta(x) = \frac{\log 2}{\log x} - \frac{\log 2 \log \log x - \log^2 2 - \log 2}{\log^2 x} + L(x), \tag{4.3} \label{4.3}
\end{equation}
where $L(x)$ is given by \eqref{1.9}. In the following lemma, we note two properties of the function $\delta(x)$, which will be useful in the proof of Theorem 
\ref{thm103}.

\begin{lem} \label{lem403}
Let $\delta(x)$ be defined as in \eqref{4.2}. Then the following two inequalities hold:
\begin{enumerate}
 \item[(a)] $\delta(x) > 0.638/\log x$ for every $x \geq \exp(\exp(3))$,
 \item[(b)] $\delta(x) < \log 2/ \log x$ for every $x \geq 230$.
\end{enumerate}
\end{lem}

\begin{proof}
Since $0.055 \log x + 0.812 > 0.638 \log \log x$ for every $x \geq 4.71 \cdot 10^8$, it follows that the inequality
\begin{displaymath}
(\log 2 - 0.638)\log x + (1+0.638)\log 2 - \frac{1.472 \cdot 3 + 2.51 - 0.638\log 2}{e^3} > 0.638 \log \log x
\end{displaymath}
holds for every $x \geq 4.71 \cdot 10^8$. The function $t \mapsto \log \log t/ \log t$ is decreasing for $x \geq e^e$. Hence
\begin{displaymath}
(\log 2 - 0.638)\log x + (1+0.638)\log 2 - \frac{1.472 \log \log x + 2.51 - 0.638\log 2}{\log x} > 0.638 \log \log x
\end{displaymath}
for every $x \geq \exp(\exp(3))$. Now it suffices to note that the last inequality is equivalent to $\delta(x) > 0.638/\log x$. This proves (a). Next, we prove 
(b). Since $\log 2 \log \log x > \log 2 + \log^22$ for every $x \geq 230 \geq \exp(\exp(1+\log 2))$, we obtain that the inequality
\begin{displaymath}
\log 2 + \log^2 2 < \log 2 \log \log x + \frac{1.472  + 2.51 - \log^2 2}{\log x}
\end{displaymath}
holds for $x \geq 230$. Again, it suffices to note that the last inequality is equivalent to $\delta(x) < \log 2/\log x$.
\end{proof}

Finally, we give the proof of Theorem \ref{thm103}.

\begin{proof}[Proof of Theorem \ref{thm103}]
First, we consider the case where $n$ is a positive integer with $n \geq 528\,491\,312 \geq \exp(\exp(3))$. By \eqref{1.3} and \eqref{1.6}, we have $2n < 
\pi(R_n) < 2.6n$. Further, $\pi(R_n) > 2n \geq 1\,056\,982\,624$ and $\pi(R_n)-n > 528\,491\,312$. Applying Lemma \ref{lem201} and Lemma \ref{lem202}, we get 
$F_2(\pi(R_n)) < p_\pi(R_n) - 2 p_{\pi(R_n)-n+1}$, where $F_2$ is defined as in \eqref{4.1}. Note that $R_n = p_{\pi(R_n)}$. Hence, by Lemma \ref{lem401}, we 
get
\begin{equation}
F_2(\pi(R_n)) < p_{\pi(R_n)} - 2 p_{\pi(R_n)-n+1} < 0. \tag{4.4} \label{4.4}
\end{equation}
In the following, we use, for convenience, the notation $\delta = \delta(n)$ and write $\beta = 2n(1+ \delta)$. So, by \eqref{4.3}, we need to prove that $\beta 
< \pi(R_n)$. For this purpose, we first show that $F_2(\beta) > 0$. From Lemma \ref{lem403}, it follows that $2n < \beta < 2.6n$. Furthermore, we have
\begin{equation}
\frac{F_2(\beta)}{2n} = (1+ \delta)\log 2 - \delta\log n - \frac{\log n}{n} + \delta + \frac{1}{n} + A_2 + B_2 + C_2 + D_2 - 
\frac{0.841(1+\delta)}{2\log^2(2n+2n\delta)}, \tag{4.5} \label{4.5}
\end{equation}
where the quantities $A_2$, $B_2$, $C_2$ and $D_2$ are given by
\begin{align*}
A_2 & = (1+ \delta)\log(1+ \delta) - \left(1+2\delta + \frac{1}{n} \right)\log \left( 1+2\delta + \frac{1}{n} \right), \\
B_2 & = (1+ \delta)\log \log(2n+ 2n\delta) - \left( 1+2\delta + \frac{1}{n} \right) \log \log (n+2n\delta + 1), \\
C_2 & = (1+ \delta) \frac{\log\log(2n+2n\delta) - 2}{\log(2n+2n\delta)} - \left( 1+2\delta + \frac{1}{n} \right)\frac{\log\log(n+2n\delta+1) - 
2}{\log(n+2n\delta + 1)}, \\
D_2 & = - (1+ \delta) \frac{(\log \log(2n+2n\delta))^2 - 6 \log \log(2n+2n\delta) + 10.667}{2 \log^2(2n+2n\delta)} \\ 
& \p{\q\q} + \left( 1+2\delta + \frac{1}{n} \right) \frac{(\log \log(n+2n\delta+1))^2 - 6 \log \log(n+2n\delta+1) + 10.667}{2 \log^2(n+2n\delta+1)}.
\end{align*}
To show that $F_2(\beta) > 0$, we give in the following some lower bounds for the quantities $A_2$, $B_2$, $C_2$ and $D_2$. To find a lower bound for $A_2$, we 
consider the function $f: (0, \infty) \to \R$ defined by $f(x) = x \log x$. Then $A_2 = f(1+\delta) - f(1+2\delta + 1/n)$. By the mean value theorem, there 
exists $\xi \in (1+\delta, 1+2\delta + 1/n)$, so that $A_2 = - (\delta+1/n)(\log \xi + 1)$. Since $\log \xi \leq \log (1+2\delta + 1/n) \leq 2\delta + 1/n$, we 
get
\begin{displaymath}
A_2 \geq - \delta - 2\delta^2 - \frac{1}{n} \left( 1 + 3 \delta + \frac{1}{n} \right).
\end{displaymath}
Applying Lemma \ref{lem403}(b) to the last inequality, we obtain that
\begin{equation}
A_2 \geq - \delta - \frac{2\log^22}{\log^2 n} - \frac{1}{\log^2n} \frac{(1+3\delta + 1/n) \log^2n}{n} \geq - \delta - \frac{0.961}{\log^2 n}. \tag{4.6} 
\label{4.6}
\end{equation}
Our next goal is to estimate $B_2$. For this purpose, we use the right-hand side inequality of \eqref{3.8}, Lemma \ref{lem403}(b) and the inequality 
$1/(x\log x) < 0.0037/\log^2x$, which holds for every $x \geq 2036$, to get
\begin{equation}
\log \log (n+2n\delta+1) 
< \log \log n + \frac{2\log 2}{\log^2 n} + \frac{1}{n \log n} < \log \log n + \frac{1.39}{\log^2 n}. \tag{4.7} \label{4.7}
\end{equation}
On the other hand, we have
\begin{displaymath}
\log \log (2n+2n\delta) = \log \log n + \log \left( 1 + \frac{\log 2 + \log(1+\delta)}{\log n} \right).
\end{displaymath}
Applying the left-hand side inequality of \eqref{3.8}, we obtain that
\begin{displaymath}
\log \log (2n+2n\delta) \geq \log \log n + \frac{\log 2}{\log n} + \frac{\log(1+\delta)}{\log n} - \frac{(\log2 + \log(1+\delta))^2}{2\log^2n}.
\end{displaymath}
Combined with
\begin{displaymath}
(\log2 + \log(1+\delta))^2 \leq ( \log2 + \delta)^2 \leq \left( \log2 + \frac{\log 2}{\log n} \right)^2 \leq 0.53,
\end{displaymath}
it follows that the inequality
\begin{displaymath}
\log \log (2n+2n\delta) \geq \log \log n + \frac{\log 2}{\log n} + \frac{\log(1+\delta)}{\log n} - \frac{0.265}{\log^2n}
\end{displaymath}
holds. Again, we use the left-hand side inequality of \eqref{3.8} to establish
\begin{displaymath}
\log \log (2n+2n\delta) \geq \log \log n + \frac{\log 2}{\log n} + \frac{\delta - \delta^2/2}{\log n} - \frac{0.265}{\log^2n}.
\end{displaymath}
Now we apply Lemma \ref{lem403}(a) and Lemma \ref{lem403}(b) to obtain that
\begin{displaymath}
\log \log (2n+2n\delta) \geq \log \log n + \frac{\log 2}{\log n} + \frac{0.361}{\log^2 n}.
\end{displaymath}
Together with the definition of $B_2$ and \eqref{4.6}, we get
\begin{displaymath}
B_2  \geq - \delta \log \log n + \frac{(1+\delta)\log 2}{\log n} - \frac{\log \log n}{n} - \frac{1.029 + 2.419\delta}{\log^2n} - \frac{1.39}{n\log^2 n}.
\end{displaymath}
Finally, we use a computer and Lemma \ref{lem403}(b) to get
\begin{equation}
B_2  \geq - \delta \log \log n + \frac{(1+\delta )\log 2}{\log n} - \frac{1.113}{\log^2 n}. \tag{4.8} \label{4.8}
\end{equation}
Next, we find an lower bound for $C_2$. For this, we apply the inequality
\begin{displaymath}
\frac{2(1+2\delta+1/n)}{\log(n+2n\delta+1)} \geq \frac{2(1+\delta)}{\log(2n+2n\delta)}
\end{displaymath}
to the definition of $C_2$ to get
\begin{displaymath}
C_2 \geq (1+\delta) \frac{\log \log (2n+2n\delta)}{\log(2n+2n\delta)} - \left( 1+2\delta + \frac{1}{n} \right) \frac{\log \log 
(n+2n\delta+1)}{\log(n+2n\delta+1)}.
\end{displaymath}
We use $2n+2n\delta \geq n+2n\delta+1 \geq n$ to obtain that the inequality
\begin{displaymath}
C_2 \geq -\log \log (n+2n\delta+1) \frac{(\delta + 1/n)\log n + (1+2\delta+1/n)(\log 2 + \log(1+\delta))}{\log(2n+2n\delta)\log(n+2n\delta+1)}
\end{displaymath}
holds. Applying the right-hand side inequality of \eqref{3.8} and Lemma \ref{lem403}(b) to the last inequality, we get
\begin{displaymath}
C_2 \geq -\log \log (n+2n\delta+1) \frac{(\log 2/\log n + 1/n)\log n + (1+2\log2/\log n+1/n)(\log 2 + \log2/\log n)}{\log(2n+2n\delta)\log(n+2n\delta+1)}.
\end{displaymath}
A computation shows that
\begin{displaymath}
\left( 1 + \frac{2\log2}{\log n} + \frac{1}{n} \right) \left( \log 2 + \frac{\log2}{\log n} \right) \leq 0.778.
\end{displaymath}
Hence
\begin{displaymath}
C_2 \geq -\frac{( \log 2 + 0.778) \log \log (n+2n\delta+1)}{\log(2n+2n\delta)\log(n+2n\delta+1)} - \frac{\log n \log 
\log(n+2n\delta+1)}{n\log(2n+2n\delta)\log(n+2n\delta+1)}.
\end{displaymath}
Note that the function $t \mapsto \log \log t/\log t$ is a decreasing function for every $t > e^e$, we obtain that
\begin{equation}
C_2 \geq -\frac{( \log 2 + 0.778) \log \log n}{\log^2n} - \frac{\log \log n}{n\log n} \geq - \frac{1.472\log \log n}{\log^2 n}. \tag{4.9} \label{4.9}
\end{equation}
Finally, we estimate $D_2$. For this purpose, we consider the function $f : (1, \infty) \to \R$ defined by
\begin{displaymath}
f(x) = \frac{(\log \log x)^2 - 6 \log \log x + 10.667}{2 \log^2 x}.
\end{displaymath}
Note that $f(x)$ is a strictly decreasing function on the interval $(1,\infty)$ and the numerator of $f(x)$ is positive for every real $x > 1$. Together with 
$2n+2n\delta \geq n+2n\delta + 1 \geq n$, we get
\begin{equation}
D_2 \geq \left( \delta + \frac{1}{n} \right) \frac{(\log \log n)^2 - 6 \log \log n + 10.667}{2 \log^2 n} > 0. \tag{4.10} \label{4.10}
\end{equation}
Finally, we combine \eqref{4.5} with \eqref{4.6} and \eqref{4.8}-\eqref{4.10} to get that the inequality
\begin{align*}
\frac{F_2(\beta)}{2n} & >
(1+ \delta)\left( \log 2 + \frac{\log 2}{\log n} \right)- \delta\log n - \frac{\log n - 1}{n} - \frac{1.472\log \log n + 2.4945}{\log^2 n} - \delta \log 
\log n - \frac{0.841\delta}{2\log^2n} \\
& \geq 
\delta \left( - \log n - \log \log n + \log 2 + \log 2/\log n \right) + \log 2 - \frac{1.472\log \log n + 2.51}{\log^2 n} + \frac{\log 2}{\log n}
\end{align*}
holds. Now it suffices to use \eqref{4.2} to get that the right-hand side of the last inequality is equal to $0$ and it follows that $F_2(\beta) > 0$. Together 
with $2n < \pi(R_n), \beta < 2.6n$, the inequality \eqref{4.4} and Proposition \ref{prop402}, we obtain that $\pi(R_n) > \beta$ for every positive integer $n 
\geq 528\,491\,312$. We conclude by direct computation.
\end{proof}

Since $L(x) \geq 0$ for every $x \geq 10^{57}$, we use Theorem \ref{thm103} to get the following weaker but more compact lower bound for $\pi(R_n)$.

\begin{kor}
Let $n$ be a positive integer satisfying $n \geq 10^{57}$. Then
\begin{displaymath}
\pi(R_n) > 2n \left( 1 + \frac{\log 2}{\log n} - \frac{\log 2 \log \log n - \log^22 - \log 2}{\log^2n} \right).
\end{displaymath}
\end{kor}

In the next corollary, we use Theorem \ref{thm103} to find that the lower bound for $\pi(R_n)$ given in Proposition \ref{prop101} also holds for every positive 
integer 
$n$ satisfying $51\,396\,214\,158\,824 \leq n \leq 10^{300}$.

\begin{kor}
Let $n$ be a positive integer satisfying $n \geq 51\,396\,214\,158\,824$. Then
\begin{displaymath}
\pi(R_n) > 2n \left( 1 + \frac{\log 2}{\log n} - \frac{\log 2 \log \log n - \log^22 - \log 2 + 0.11}{\log^2n} \right).
\end{displaymath}
\end{kor}

\begin{proof}
The claim follows directly by Theorem \ref{thm103} and the fact that $L(x) \geq - 0.11/\log^2x$ for every $x \geq 51\,396\,214\,158\,824$.
\end{proof}

Finally, we give the following result concerning a lower bound for $\pi(R_n)$.

\begin{kor}
Let $n$ be a positive integer satisfying $n \geq 85$. Then
\begin{displaymath}
\pi(R_n) > 2n \left( 1 + \frac{\log 2}{\log n} - \frac{\log 2 \log \log n}{\log^2n} \right).
\end{displaymath}
\end{kor}

\begin{proof}
Since $L(x) + (\log^22 + \log 2)/\log^2x \geq 0$ for every $x \geq 20$, we apply Theorem \ref{thm103} to get the correctness of the corollary for every 
positive integer $n \geq 1\,245$. We conclude by direct computation.
\end{proof}

\section{Proof of Theorem \ref{thm107}}

In this section we give a proof of Theorem \ref{thm107} by using Theorem 3.22 of \cite{axler2016}. For this, we need to introduce the following notations. By 
\cite[Corollary 3.4 and Corollary 3.5]{axler20162}, we have
\begin{equation}
\frac{x}{\log x - 1 - \frac{1}{\log x}} < \pi(x) < \frac{x}{\log x - 1 - \frac{1.17}{\log x}}, \tag{5.1} \label{5.1}
\end{equation}
where the left-hand side inequality is valid for every $x \geq 468\,049$ and the right-hand side inequality holds for every $x \geq 5.43$. Using the
right-hand side inequality of \eqref{5.1}, we get $p_n > n(\log p_n - 1 - 1.17/\log p_n)$ for every positive integer $n$. In addition, we set $\e > 0$ and 
$\lambda = \e/2$. Let $S = S(\e)$ be defined by
\begin{displaymath}
S = \exp \left( \sqrt{ 1.17 + \frac{2(1+\e)}{\e} \left( 0.17 + \frac{\log 2}{\log(2 \cdot 5.43)}\right) + \left( \frac{1}{2} + \frac{(1+\e)\log 2}{\e} 
\right)^2} + \frac{1}{2} + \frac{(1+\e)\log 2}{\e} \right)
\end{displaymath}
and let $T = T(\e)$ be defined by $T = \exp(1/2 + \sqrt{1.17 + 0.17/\lambda + 1/4})$. By setting $X_9 = X_9(\e) = \max \{ 468\,049, 2S, T\}$, we get 
the following result.

\begin{lem} \label{lem501}
Let $\e > 0$. For every positive integer $n$ satisfying $n \geq (\pi(X_9) + 1)/(2(1+ \e))$, we have
\begin{displaymath}
R_n \leq p_{\lceil 2(1+ \e)n \rceil}.
\end{displaymath}
\end{lem}

\begin{proof}
This follows from Theorem 3.22 and Lemma 3.23 of \cite{axler2016}.
\end{proof}

The following proof of Theorem \ref{thm107} consists of three steps. In the first step, we apply Theorem \ref{thm102} and Theorem \ref{thm103} to derive a 
lower bound for the quantity $m \pi(R_n) - \pi(R_{mn})$, which holds for every positive integers $m$ and $n$ satisfying $m \geq 2$ and $n \geq \max \{ \lceil 
5225/m \rceil, 1\,245 \}$. Then, in the second step, we use this lower bound and a computer to establish Theorem \ref{thm107} for the cases $m=2$ and $m \in 
\{3, 4, \ldots, 19 \}$. Finally, we consider the case where $m \geq 20$. In this case, we first show that the inequality $\pi(R_{mn}) \leq m \pi(R_n)$ holds 
for every positive integer $n \geq 1\,245$. So it suffices to show that the required inequality also holds for every positive integers $m$ and  $n$ with $m 
\geq 20$ and $N(m) \leq n \leq 1\,244$, where $N(m)$ is defined as in Theorem \ref{thm107}, with the only exception $(m,n) = (38,9)$. For this purpose, note 
that
\begin{equation}
\pi(R_{mn}) \leq m \pi(R_n) \q \Leftrightarrow \q R_{mn} \leq p_{m\pi(R_n)}. \tag{5.2} \label{5.2}
\end{equation}
Now, for each $n \in \{ 2, \ldots, 1\,244\}$ we use \eqref{5.2} and Lemma \ref{lem501} with $\e = \pi(R_n)/2n - 1$ (note that $\e > 0$ by \eqref{1.3}) to find 
a positive integer $M(n)$, so that $R_{mn} \leq p_{m\pi(R_{n})}$ for every positive integer $m \geq M(n)$. Finally we check with a computer for which $m < M(n)$ 
the inequality $R_{mn} \leq p_{m\pi(R_{n})}$ holds.

\begin{proof}[Proof of Theorem \ref{thm107}]
First, we note that the inequality \eqref{1.10} holds for $m = 1$. So, we can assume that $m \geq 2$. Let $n$ be a positive integer with $n \geq \max \{ \lceil 
5225/m \rceil, 1\,245 \}$. By \eqref{3.4}, \eqref{3.5} and Theorem \ref{thm102}, we have
\begin{equation}
\pi(R_{mn}) < 2mn \left( 1 + \frac{\log 2 + \log 2/\log(mn) + 0.565/\log^2(mn)}{\log(mn) + \log \log(mn) - \log 2 - \log 2/\log(mn)} \right) \tag{5.3} 
\label{5.3}
\end{equation}
and, by \eqref{4.2}, \eqref{4.3} and Theorem \ref{thm103}, we have
\begin{equation}
\pi(R_n) > 2n \left( 1 + \frac{\log 2 + \log 2/\log n - (1.472\log \log n + 2.51)/\log^2 n}{\log n + \log \log n - \log 2 - \log 2/\log n} \right). \tag{5.4} 
\label{5.4}
\end{equation}
We set $\lambda(x) = \log x + \log \log x - \log 2 - \log 2/\log x$ and $\phi(x) = 1.472 \log \log x + 2.51$. Then, by \eqref{5.3} and \eqref{5.4}, we get
\begin{equation}
\frac{m\pi(R_n) - \pi(R_{mn})}{2mn} > \frac{W_m(n)}{\lambda(n)\lambda(mn)}, \tag{5.5} \label{5.5}
\end{equation}
where
\begin{align*}
W_m(n) & = \log 2 \log m + \log 2(\log \log(mn) - \log \log n) + \log 2 \left( \frac{\log(mn)}{\log n} - \frac{\log n}{\log(mn)} \right) \\
& \p{\q\q} + \log 2 \left( \frac{\log \log(mn)}{\log n} - \frac{\log \log n}{\log(mn)} \right) - \frac{\phi(n)\lambda(mn)}{\log^2 n} - 
\frac{0.565\lambda(n)}{\log^2(mn)}.
\end{align*}
Clearly, it suffices to show that $W_m(n) \geq 0$. Setting $g(x) = \log \log x$, we get, by the mean value theorem, that there exists a real number $\xi \in 
(n, mn)$ such that $g(mn) - g(n) = (m-1)ng'(\xi)$. Hence
\begin{equation}
\log \log(mn) - \log \log n = \frac{(m-1)n}{\xi \log \xi} \geq \frac{m-1}{m \log(mn)} \geq \frac{1}{2\log(mn)}. \tag{5.6} \label{5.6}
\end{equation}
Further, we have
\begin{equation}
\frac{\log(mn)}{\log n} - \frac{\log n}{\log(mn)} = \frac{\log m}{\log n} + \frac{\log m}{\log(mn)}, \tag{5.7} \label{5.7}
\end{equation}
as well as
\begin{equation}
\frac{\log \log(mn)}{\log n} - \frac{\log \log n}{\log(mn)} > \frac{\log m \log \log n}{\log^2(mn)}. \tag{5.8} \label{5.8}
\end{equation}
Combining \eqref{5.6}-\eqref{5.8} with the definition of $W_m(n)$, we obtain that the inequality
\begin{align*}
W_m(n) & > \log m \left( \log 2 + \frac{\log 2}{\log n} \right) + \log 2 \left( \frac{\log m + 1/2}{\log(mn)} + \frac{\log m \log \log 
n}{\log^2(mn)} \right) - \frac{\phi(n)\lambda(mn)}{\log^2 n} - \frac{0.565\lambda(n)}{\log^2(mn)}.
\end{align*}
Since $\lambda(x) < \log x + \log \log x - \log 2 < \log x + \log \log x$, we get
\begin{align*}
W_m(n) & > \log m \left( \log 2 + \frac{\log 2}{\log n} - \frac{\phi(n)}{\log^2 n} \right) + \log 2 \left( \frac{\log m + 1/2}{\log(mn)} + \frac{\log m \log 
\log n}{\log^2(mn)} \right) \\
& \p{\q\q} - \frac{\phi(n)}{\log n} - \frac{\phi(n)\log \log(mn)}{\log^2 n} + \frac{\phi(n) \log 2}{\log^2 n} - \frac{0.565\log n}{\log^2(mn)} - \frac{0.565 
\log \log n}{\log^2(mn)}.
\end{align*}
Now, we use the right-hand side inequality of \eqref{3.8} to get $\log \log(mn) \leq \log \log n + \log m/\log n$. Finally, we have
\begin{align}
W_m(n) & > \log m \left( \log 2 + \frac{\log 2}{\log n} - \frac{\phi(n)}{\log^2 n} - \frac{\phi(n)}{\log^3 n} \right) - \frac{\phi(n)}{\log n} - 
\frac{\phi(n)(\log \log n - \log 2)}{\log^2 n} \tag{5.9} \label{5.9} \\
& \p{\q\q} + \frac{(\log m + 1/2) \log 2 - 0.565}{\log(mn)} + \frac{(\log m \log 2 - 0.565)\log \log n}{\log^2(mn)} \nonumber
\end{align}
for every positive integers $m$ and $n$ satisfying $m \geq 2$ and $n \geq \max \{ \lceil 5225/m \rceil, 1\,245 \}$. Next, we use this inequality to prove the 
theorem. For this purpose, we consider the following three cases:
\begin{enumerate}
 \item[(i)] \textit{Case} 1: $m = 2$. \newline
First, let $n \geq 4\,903\,689$. In this case, we have $(\log m + 1/2) \log 2 - 0.565 \geq 0.262$ and $\log m \log 2 - 0.565 > - 0.085$. Hence
\begin{displaymath}
\frac{(\log m + 1/2) \log 2 - 0.565}{\log(mn)} + \frac{(\log m \log 2 - 0.565)\log \log n}{\log^2(mn)} > 0.
\end{displaymath}
Applying this inequality to \eqref{5.9}, we get
\begin{displaymath}
W_2(n) > \log 2 \left( \log 2 + \frac{\log 2}{\log n} - \frac{\phi(n)}{\log^2 n} - \frac{\phi(n)}{\log^3 n} \right) - \frac{\phi(n)}{\log n} - 
\frac{\phi(n)(\log \log n - \log 2)}{\log^2 n}.
\end{displaymath}
Since $\log 2 - \phi(x)/\log x - \phi(x)/\log^2 x > 0$ for every real $x \geq 10\,377$, we get
\begin{displaymath}
W_2(n) > \log^2 2 - \frac{\phi(n)}{\log n} - \frac{\phi(n)(\log \log n - \log 2)}{\log^2 n}.
\end{displaymath}
Note that the right-hand side of the last inequality is positive. Combined with \eqref{5.5}, we get that $\pi(R_{2n}) \leq 2\pi(R_n)$ holds for every positive 
integer $n \geq 4\,903\,689$. A direct computation shows that the inequality $\pi(R_{2n}) \leq 2\pi(R_n)$ also holds for every positive integer $n$ so that $ 
1\,245 \leq n \leq 4\,903\,689$. 
 \item[(ii)] \textit{Case} 2: $m \in \{ 3, 4, \ldots, 19 \}$. \newline
First, we consider the case where $n \geq 6\,675$. By \eqref{5.9}, we have
\begin{displaymath}
W_m(n) > \log m \left( \log 2 + \frac{\log 2}{\log n} - \frac{\phi(n)}{\log^2 n} - \frac{\phi(n)}{\log^3 n} \right) - \frac{\phi(n)}{\log n} - 
\frac{\phi(n)(\log \log n - \log 2)}{\log^2 n}.
\end{displaymath}
We set $\delta_2 = 0.003314$ to obtain that the inequality
\begin{displaymath}
\delta_2 + \frac{\log 2}{\log x} - \frac{\phi(x)}{\log^2 x} - \frac{\phi(x)}{\log^3 x} > 0
\end{displaymath}
holds for every real $x \geq 6\,675$. So we see that
\begin{displaymath}
W_m(n) > (\log 2 - \delta_2) \log 3 - \frac{\phi(n)}{\log n} - \frac{\phi(n)(\log \log n - \log 2)}{\log^2 n}
\end{displaymath}
and since the right-hand side of the last inequality is positive, we use \eqref{5.5} to conclude that $\pi(R_{mn}) \leq m\pi(R_n)$ holds for each $m \in 
\{ 3, 4, \ldots, 19 \}$ and every positive integer $n \geq 6\,675$. For $m \in \{ 3,4 \}$, we verify with a direct computation that the inequality $\pi(R_{mn}) 
\leq m\pi(R_n)$ also holds for every positive integer $n$ so that $189 \leq n \leq 6\,674$. For $m \in \{ 5,6 \}$, we use a computer to check that the 
inequality $\pi(R_{mn}) \leq m\pi(R_n)$ is also valid for every positive integer $n$ satisfying $85 \leq n \leq 6\,674$. Finally, if $m \in \{ 7, 8, 
\ldots, 19 \}$, a computer check shows that the required inequality also holds for every positive integer $n$ with $10 \leq n \leq 6\,674$.
 \item[(iii)] \textit{Case} 3: $m \geq 20$. \newline
First, let $n \geq 1\,245$. Setting $\delta_3 = 0.03$, we obtain, similar to Case 2, that
\begin{displaymath}
W_m(n) > (\log 2 - \delta_3) \log 20 - \frac{\phi(n)}{\log n} - \frac{\phi(n)(\log \log n - \log 2)}{\log^2 n}.
\end{displaymath}
Note that the right-hand side of the last inequality is positive. Together with \eqref{5.5}, we get that $\pi(R_{mn}) \leq m\pi(R_{n})$ holds for all positive 
integers $m$ and $n$ satisfying $m \geq 20$ and $n \geq 1\,245$. Now, for each $n \in \{ 2, \ldots, 1\,244 \}$, we use \eqref{5.2}, Lemma \ref{lem501} with $\e 
= \pi(R_n)/2n - 1$ and a C++ version of the following MAPLE code to find positive integer $M(n) \geq 20$, so that $R_{mn} \leq p_{m\pi(R_{n})}$ for every 
positive integer $m \geq M(n)$ and then we check for which $m$ with $20 \leq m < M(n)$ the inequality $R_{mn} \leq p_{m\pi(R_{n})}$ holds:
\begin{align*}
& > \texttt{restart: with(numtheory): Digits := 100:} \vspace{4mm} \\
& > \texttt{for n from 1244 by -1 to 2 do} \\
& \p{>>>} \texttt{ep := pi(R[n])/(2*n)-1: \# R[n] denotes the nth Ramanujan prime} \\
& \p{>>>} \texttt{lambda := ep/2:} \\
& \p{>>>} \texttt{S := ceil(evalf(exp(sqrt(1.17+2*(1+ep)/ep*(0.17+log(2)/log(2*5.43))+} \\
& \p{>>>>>>} \texttt{(1/2+(1+ep)*log(2)/ep)$\widehat{\p{a}}$2)+1/2+(1+ep)*log(2)/ep):} \\
& \p{>>>} \texttt{T := ceil(evalf(exp(sqrt(1.17+0.17/lambda+1/4)+1/2))):} \\
& \p{>>>} \texttt{X9 := max(468049,2*S,T): M := ceil((1+pi(X9))/(2*(1+ep))):} \\
& \p{>>>} \texttt{\# Hence pi(R[mn]) <= m*pi(R[n]) for all m >= M by Lemma 5.1} \\
& \p{>>>} \texttt{while M*pi(R[n]) - pi(R[n*M]) >= 0 and M >= 20 do} \\
& \p{>>>>>} \texttt{M := M-1:} \\
& \p{>>>} \texttt{end do:} \\
& \p{>>>} \texttt{L[n] := M+1:} \\
& \p{>} \texttt{ end do:}
\end{align*}
Since $L[i] = 20$ for every $i \in \{2, \ldots, 1244 \} \setminus \{ 9 \}$ and $L[9] = 39$, we get that $\pi(R_{mn}) \leq m\pi(R_n)$ for every positive 
integers $n, m$ with $n \in \{2, \ldots, 1244 \} \setminus \{ 9\}$ and $m \geq 20$ and for every positive integers $n,m$ with $n = 9$ and $m \geq 39$. A direct 
computation shows that the inequality $\pi(R_{9m}) \leq m\pi(R_9)$ holds for every $m$ with $20 \leq m \leq 37$ as well and that $38\pi(R_9) - \pi(R_{9 \cdot 
38}) = -2$.
\end{enumerate}
So, we showed that the inequality $\pi(R_{mn}) \leq m\pi(R_n)$ holds for every $m \in \N$ and every positive integer $n \geq N(m)$ with the only exception 
$(m,n) = (38,9)$, as desired.
\end{proof}

%

%

We use Theorem \ref{thm107} and a computer to get the following remark.

\begin{rema}
The inequality $\pi(R_{mn}) \leq m\pi(R_n)$ fails if and only if $(m,n) \in \N_{\geq 2} \times \{ 1 \}$ (see \eqref{1.3}) or
\begin{align*}
(m,n) & \in \{ (2,3), (2,7), (2,8), (2,9), (2,22), (2,23), (2,25), (2,37), (2,38), (2,49), (2,53), (2,54), \\
& \p{ \in \{ , } (2,55), (2,66), (2,82), (2,83), (2,84), (2,85), (2,86), (2,87), (2,101), (2,102), (2,113), \\ 
& \p{ \in \{ , } (2,114), (2,115), (2,160), (2,161), (2,162), (2,179), (2,180), (2,184), (2,185), (2,186), \\
& \p{ \in \{ , } (2,232), (2,240), (2,241), (2,246), (2,247), (2,376), (2,377), (2,378), (2,379), (2,380), \\
& \p{ \in \{ , } (2,381), (2,386), (2,387), (2,388), (2,412), (2,531), (2,532), (2,537), (2,538), (2,547), \\
& \p{ \in \{ , } (2,548), (2,549), (2,550), (2,551), (2,552), (2,553), (2,554), (2,555), (2,556), (2,557), \\
& \p{ \in \{ , } (2,558), (2,792), (2,793), (2,794), (2,795), (2,796), (2,797), (2,798), (2,799), (2,800), \\
& \p{ \in \{ , } (2,801), (2,802), (2,803), (2,804), (2,1140), (2,1141), (2,1142), (2,1146), (2,1147), \\
& \p{ \in \{ , } (2,1202), (2,1241), (2,1242), (2,1243), (2,1244), (3,9), (3,11), (3,23), (3,25), (3,49), \\
& \p{ \in \{ , } (3,54), (3,55), (3,56), (3,57), (3,66), (3,67), (3,83), (3,84), (3,114), (3,115), (3,160), \\
& \p{ \in \{ , } (3,187), (3,188), (4,9), (4,11), (4,37), (4,38), (4,42), (4,54), (4,55), (4,82), (4,83), \\
& \p{ \in \{ , } (4,84), (4,114), (4,115), (4,188), (5,3), (5,9), (5,84), (6,28), (6,54), (6,55), (6,84), \\
& \p{ \in \{ , } (7,3), (7,9), (8,9), (9,9), (10,9), (11,3), (11,9), (12,9), (13,9), (14,9), (15,3), (15,9), \\
& \p{ \in \{ , } (16,9), (17,9), (18,9), (19,9), (38,9) \}.
\end{align*}
\end{rema}

%

\end{document}